\documentclass{article}

\newcommand*\xoverline[2][0.75]{%
	\sbox{\myboxA}{$\m@th#2$}%
	\setbox\myboxB\null
	\ht\myboxB=\ht\myboxA%
	\dp\myboxB=\dp\myboxA%
	\wd\myboxB=#1\wd\myboxA
	\sbox\myboxB{$\m@th\overline{\copy\myboxB}$}
	\setlength\mylenA{\the\wd\myboxA}
	\addtolength\mylenA{-\the\wd\myboxB}%
	\ifdim\wd\myboxB<\wd\myboxA%
	\rlap{\hskip 0.5\mylenA\usebox\myboxB}{\usebox\myboxA}%
	\else
	\hskip -0.5\mylenA\rlap{\usebox\myboxA}{\hskip 0.5\mylenA\usebox\myboxB}%
	\fi}
\makeatother

\usepackage{amsmath,amsthm,amsfonts,amssymb,amscd} 
\usepackage{bm} 
\usepackage[top=2.54cm,bottom=2.54cm,left=3.17cm,right=3.17cm]{geometry}  
\usepackage{esint,mathtools,bbm} 
\usepackage{tikz}
\usetikzlibrary{backgrounds}
\usepackage{float} 
\usepackage{algorithm,algorithmicx,algpseudocode}  
\usepackage{booktabs} 
\usepackage{hyperref}
\usepackage{cleveref}
\crefformat{equation}{(#2#1#3)}

\usepackage{xcolor}
\definecolor{mycolor1}{RGB}{237,227,135}
\definecolor{mycolor2}{RGB}{237,237,237}
\definecolor{mycolor3}{RGB}{59,32,12}
\definecolor{mycolor4}{RGB}{222,129,0}
\definecolor{mycolor5}{RGB}{204,252,98}

\newtheorem{theorem}{Theorem}[section]
\newtheorem{lemma}[theorem]{Lemma}
\newtheorem{proposition}[theorem]{Proposition}

\theoremstyle{definition}
\newtheorem{definition}[theorem]{Definition}

\newtheorem{problem}{Problem}
\crefname{problem}{problem}{problems}

\newtheorem{assumption}{Assumption}
\newtheorem{corollary}[theorem]{Corollary}

\theoremstyle{remark}
\newtheorem{remark}[theorem]{Remark}

\numberwithin{equation}{section}

\newcommand{\Z}{\mathbb{Z}}
\newcommand{\R}{\mathbb{R}}

\newcommand{\Brackets}[1]{\left( #1 \right)}
\newcommand{\SquareBrackets}[1]{\left[ #1\right]}
\newcommand{\Braces}[1]{\left\{ #1\right\}}
\newcommand{\FuncAction}[2]{\left\langle #1,#2 \right\rangle}
\newcommand{\Norm}[1]{\left\lVert #1 \right\rVert}
\newcommand{\Seminorm}[1]{\left\lvert #1 \right\rvert}
\newcommand{\vertiii}[1]{{\left\vert\kern-0.25ex\left\vert\kern-0.25ex\left\vert #1\right\vert\kern-0.25ex\right\vert\kern-0.25ex\right\vert}}
\newcommand{\dx}{\mathrm{d}}
\newcommand{\overbar}[1]{\mkern 1.5mu\overline{\mkern-1.5mu#1\mkern-1.5mu}\mkern 1.5mu} 

\newcommand{\Div}{\mathrm{div}}

\title{Convergence Rate of Multiscale Finite Element Method for Various Boundary Problems}
\author{Changqing Ye (yechangqing@lsec.cc.ac.cn)\footnote{LSEC, ICMSEC, Academy of Mathematics and Systems Science, Chinese Academy of Sciences, Beijing 100190, China, and School of Mathematical Sciences, University of Chinese Academy of Sciences, Beijing 100049, China} \\Dong hao (donghao@mail.nwpu.edu.cn)\footnote{School of Mathematics and Statistics, Xidian University, Xi’an 710071, PR China} \\Junzhi Cui (cjz@lsec.cc.ac.cn)\footnote{LSEC, ICMSEC, Academy of Mathematics and Systems Science, Chinese Academy of Sciences, Beijing 100190, China}}
\date{\today \footnote{Available \href{https://doi.org/10.1016/j.cam.2020.112754}{online}}}

\begin{document}

\maketitle
\begin{abstract}
In this paper, we examine the effectiveness of classic multiscale finite element method (MsFEM) (Hou and Wu, 1997; Hou et al., 1999) for mixed Dirichlet-Neumann, Robin and hemivariational inequality boundary problems. Constructing so-called boundary correctors is a common technique in existing methods to prove the convergence rate of MsFEM, while we think not reflects the essence of those problems. Instead, we focus on the first-order expansion structure. Through recently developed estimations in homogenization theory, our convergence rate is provided with milder assumptions and in neat forms. 
\end{abstract}

\section{Introduction}
Multiscale problems are ubiquitous in science and engineering, and the most common manifestation is PDEs with highly oscillatory coefficients. For examples, physical field equations for modeling composite material, and Darcy's flows in porous media. Directly solving those problems by universal methods such as finite element method (FEM) is still impractical even with a huge development on computational power. Therefore, the purposes of multiscale computation are exploiting advanced computer architecture, capturing small-scale information while keeping the main workload bearable.

It has been decades since Hou et al. developed a multiscale finite element method (cf. \cite{Hou1997,Hou1999}). Simply put, MsFEM solves a multiscale problem on a coarse mesh while fine-scale properties are revealed in finite element bases, here ``coarse'' is comparing to the original characteristic scale. Hence, a great reduction in total freedoms will be achieved. The construction of MsFEM basis functions is through locally solving PDEs which are determined by the leading order differential operator. Originally, the boundary conditions for those bases are linear to guarantee conformity, later an over-sampling technique together with nonconforming FEM numerical analysis was introduced to weaken the influence of so-called resonance error (see \cite{Efendiev2000,YalchinEfendiev2009}). Recently, a generalized MsFEM performed a further step (\cite{Efendiev2013}), in this method multiscale basis functions are designed by a two-stage process, and this method has been successfully applied to high contrast flow problems (\cite{Chung2014}).

While MsFEM has been broadly applied to different scenarios, to our knowledge, there are several numerical analysis results remaining unproven, and this is what our paper mainly concerns. Specifically, an estimation for MsFEM solutions on Dirichlet problems was given in \cite{zhimingchenhaijunwu1991} and states that the error is
\[
C\Brackets{h+\epsilon}\Norm{f}_{L^2(\Omega)}+C\Brackets{\sqrt{\frac{\epsilon}{h}}+\sqrt{\epsilon}}\Norm{u_0}_{W^{1,\infty}(\Omega)}.
\]
However, a bounded gradient assumption for the homogenized solution $u_0$ seems too strong. The proof relies on a boundary corrector $\theta_\epsilon$, which is intuitive to get constructed on Dirichlet problems while slightly complicate on Neumann problems. Chen and Hou did show a proper way in \cite{Chen2002} to define boundary corrector for pure Neumann problems. Nevertheless, their method is highly technical and strict for the homogenized solution, also difficult to extend to Robin or nonlinear boundary problems.

Due to recently developed theories for periodic homogenization \cite{Shen2016}, we can directly build the estimation from variational form, rather than construct boundary correctors first. This thinking also motivates our work \cite{Ye2019} on boundary hemivariational inequality{\textemdash}a nonlinear boundary condition that frequently appears in the frictional contact modeling. With those estimations prepared, we can now clearly and rigorously examine the effectiveness of classic MsFEM for various boundary problems.

We arrange the following sections as: in \cref{sec:pre} we first introduce the notations, general settings, and model problems, then review the homogenization theory, highlight on the estimations for later analysis, finally a concise description of MsFEM is also provided; in \cref{sec:main} we progressively prove our main results, with the best effort to keep selfconsistency; conclusion remarks are included in \cref{sec:conclusion}.

\section{Preliminaries}\label{sec:pre}
For simplicity, we consider 2D problems through the full text. We reserve $\Omega$ for a domain (bounded and open set) with Lipschitz boundary and $d$ for spacial dimension ($d=2$). The Einstein summation convention is adopted, means summing repeated indexes from $1$ to $d$. The Sobolev spaces $W^{k,p}$ and $H^{k}$ are defined as usual (see e.g., \cite{Brenner2008}) and we abbreviate the norm and seminorm of Sobolev space $H^k(\Omega)$ as $\Norm{ \cdot}_{k,\Omega}$ and $\Seminorm{\cdot}_{k,\Omega}$. 

\subsection{Homogenization Theory}
In this subsection, we review the homogenization theory and related estimations for our model problems. 
\begin{definition}[1-periodicity] 
A (vector/matrix value) function $f$ is called 1-periodic, if
\[
f(\bm{x}+\bm{z}) = f(\bm{x}) ~~~~ \forall \bm{x} \in \R^d \text{ and } \forall \bm{z} \in \Z^d .
\]
\end{definition}
A 1-periodic (vector/matrix value) function $f(\bm{x})$ with $\epsilon$ superscripted means a scaling on $\bm{x}$ as $f^\epsilon(\bm{x}) = f(\bm{x}/\epsilon)$. 
\begin{definition}
The coefficient matrix $A(\bm{x})$ is called symmetric and uniformly elliptic if
\begin{equation}
\begin{aligned}
& A_{ij}(\bm{x}) = A_{ji}(\bm{x}); \\
& \kappa_1 \Seminorm{\xi}^2 \leq A_{ij}(\bm{x})\xi_i \xi_j \leq \kappa_2 \Seminorm{\xi}^2 ~~~~ \text{for a.e. } \bm{x} \text{~ in ~} \bm{x} \text{'s domain and } \xi \in \R^d.
\end{aligned}
\end{equation}
\end{definition}

\begin{problem}[Mixed Dirichlet-Neumann problem] \label{pro:a}
Split $\partial \Omega$ into the two disjointed parts $\Gamma_D, \Gamma_N$. The problem states as:
\[
\left\{
\begin{aligned}
-\Div\Brackets{A^\epsilon \nabla u_\epsilon} = f &\text{~ in ~} \Omega \\
u_\epsilon = 0 &\text{~ on ~} \Gamma_D \\
\bm{n} \cdot A^\epsilon \nabla u_\epsilon = g &\text{~ on ~} \Gamma_N
\end{aligned}
\right. .
\]
Its variational form is:
\[
\left\{
\begin{aligned}
& \text{Find~} u_\epsilon \in H^1_{0, \Gamma_D}(\Omega), \text{~ s.t. } \forall v \in H^1_{0, \Gamma_D}(\Omega) \\
& \int_\Omega A^\epsilon \nabla u_\epsilon \cdot v = \int_\Omega f v + \int_{\Gamma_N} g v
\end{aligned}
\right. .
\]
\end{problem}

For simplicity, we assume the source term $f$ belongs to $L^2(\Omega)$, and the boundary term $g \in L^2(\Gamma_N)$.
\begin{problem}[Robin problem] \label{pro:b} 
We consider a following problem:
\[
\left\{
\begin{aligned}
-\Div\Brackets{A^\epsilon \nabla u_\epsilon} = f &\text{~ in ~} \Omega \\
\bm{n} \cdot A^\epsilon \nabla u_\epsilon + \alpha u_\epsilon = g &\text{~ on ~} \partial \Omega
\end{aligned} 
\right. ,
\]
with its variational form:
\[
\left\{
\begin{aligned}
& \text{Find~} u_\epsilon \in H^1(\Omega), \text{~ s.t. } \forall v \in H^1(\Omega) \\
& \int_\Omega A^\epsilon \nabla u_\epsilon \cdot v + \int_{\partial \Omega} \alpha u_\epsilon v= \int_\Omega f v + \int_{\partial \Omega} g v
\end{aligned}
\right. .
\]
\end{problem}
We assume $0 < \alpha_1 \leq \alpha \leq \alpha_2$ to keep the coercivity of bilinear form. On a Robin problem, an energy norm can be defined as $\vertiii{v}^2 = \int_\Omega A^\epsilon \nabla v \cdot \nabla v + \int_{\partial \Omega} \alpha v^2$, which is equivalent to $\Norm{v}_{1,\Omega}$.

Before introducing hemivariational boundary inequality problems, several notations and definitions will be provided first:
\begin{definition}[See \cite{Denkowski2003}]
Let $\varphi: X \rightarrow \R$ be a locally Lipschitz function. For $x, h \in X$, the \emph{generalized directional derivative} of $\varphi$ at $x$ along the direction $h$, denoted by $\varphi^0(x; h)$ is defined by
\[
\varphi^0(x; h) \coloneqq \limsup_{y\rightarrow x, \lambda \downarrow 0} \frac{\varphi(y+\lambda h)-\varphi(y)}{\lambda} = \inf_{\epsilon, \delta>0} \sup_{\substack{\Norm{x-y}_X<\epsilon \\ 0<\lambda< \delta}} \frac{\varphi(y+\lambda h)-\varphi(y)}{\lambda} .
\]
The \emph{generalized subdifferential} of $\varphi$ at $x \in X$, is the nonempty set $\partial \varphi (x) \subset X^*$ defined by
\[
\partial \varphi(x) \coloneqq \{ x^* \in X^*: \FuncAction{x^*}{h} \leq \varphi^0(x; h), \forall h \in X \}.
\] 
\end{definition}
On a hemivariational inequality problem, we decompose boundary $\partial \Omega$ into disjointed parts $\Gamma_D$, $\Gamma_N$ and $\Gamma_C$ to impose different conditions. Then
 
\begin{problem}[Boundary hemivariational inequality] \label{pro:c} 
Let $j$ be a locally Lipschitz function on $L^2(\Gamma_C)$, and $\gamma_j$ be the trace operator $H^1_{0, \Gamma_D}(\Omega) \mapsto L^2(\Gamma_C)$, the problem states:
\[
\left\{
\begin{aligned}
-\Div(A^\epsilon(\bm{x})\nabla u_\epsilon) = f ~~~~ &\text{in} ~~ \Omega \\
u_\epsilon = 0 ~~~~ &\text{on} ~~ \Gamma_D \\
\bm{n} \cdot A^\epsilon \nabla u_\epsilon = g ~~~~ &\text{on} ~~ \Gamma_N \\
-\bm{n} \cdot A^\epsilon \nabla u_\epsilon \in \partial j(\gamma_j u_\epsilon) ~~~~ &\text{on} ~~ \Gamma_C
\end{aligned}.
\right. 
\]
It has a following hemivariational form:
\[
\left\{
\begin{aligned}
&\text{Find } u_\epsilon \in H^1_{0,\Gamma_D}(\Omega), \text{ s.t. } \forall v \in H^1_{0,\Gamma_D}(\Omega) \\
& \int_\Omega A^\epsilon \nabla u_\epsilon \cdot \nabla v + j^0(\gamma_j u_\epsilon; \gamma_j v) \geq \int_\Omega f v+\int_{\Gamma_N} gv 
\end{aligned}
\right. .
\]
\end{problem}

To make this problem solvable, we need several assumptions (\cite{Han2019,Han2017}):
\begin{assumption} \label{ass:hemi}
There exist constants $c_0,c_1,\alpha_j$, such that:
\[
\begin{aligned}
&\Norm{x^*}_{V_j^*} \leq c_0+c_1\Norm{x}_{V_j} \forall x \in V_j, \forall x^* \in \partial j(x); \\
&j^0(x_1; x_2-x_1)+j^0(x_2; x_1-x_2) \leq \alpha_j \Norm{x_1-x_2}_{V_j}^2 \forall x_1, x_2 \in V_j .
\end{aligned}
\]
Let $c_j=\Norm{\gamma_j}_{V\rightarrow V_j}$ be the operator norm, there exists that, 
\[
\Delta \coloneqq \kappa_1 - \alpha_j c_j^2 > 0 .
\]
Here we use $V_j$, $V$ for abbreviations of $L^2(\Gamma_C)$ and $H^1_{0, \Gamma_D}(\Omega)$. 
\end{assumption}

The fascinating result of homogenization theory is that $u_\epsilon$ can converge as $\epsilon \rightarrow 0$, and this limitation is called the homogenized solution $u_0$ which satisfying a homogenized equation. One interesting fact is the coefficients corresponded to this equation is not the averages of $A(\bm{y})$ in periodic cell but elaborate values $\hat{A}_{ij}$. Denote $Q=(-1/2, 1/2)^d$ and $W_\sharp^{k,p}(Q)$ as the completion of smooth 1-periodic function by $W^{k,p}(Q)$ norm, similarly for $H^k_\sharp(Q)$.

\begin{definition}[Correctors and homogenized coefficients]
We denote $\{N_l(\bm{y}\}_{l=1}^d$ as correctors for $A^\epsilon$, and correctors satisfy a group of PDEs with periodic boundary conditions:
\[
\left\{
\begin{aligned}
& -\Div\Brackets{A(\bm{y})\nabla N_l}=\Div (Ae_l)=\partial_i A_{il}(\bm{y}) ~~~~ \text{ in }Q\\
& N_l(\bm{y}) \in H^1_\sharp(Q) \text{ and } \int_Q N_l = 0
\end{aligned}
\right. .
\]
The homogenized coefficients are defined by
\[
\hat{A}_{il}=\frac{1}{\Seminorm{Q}}\int_Q \Brackets{A_{il}+A_{ij}\partial_j N_l} \dx \bm{y} \coloneqq \fint_Q \Brackets{A_{il}+A_{ij}\partial_j N_l} \dx \bm{y} .
\]
\end{definition} 
We have a theorem:
\begin{theorem}[Proofs in \cite{Jikov1994,Shen2016,Ye2019}]
$u_\epsilon$ converges to $u_0$ \emph{weakly} in \cref{pro:a,pro:b,pro:c} As $\epsilon \rightarrow 0$, and $u_0$ is determined by the same equations (variational or hemivariational forms) with substituting $A^\epsilon(\bm{x})$ with $\hat{A}$.  
\end{theorem}

From the point of computation, weak convergence is inadequate because it does not provide quantitative information. Hence an asymptotic expansion is developed and successfully applied to various problems (\cite{Allaire2002,Bensoussan2011}). Here we only consider a first-order expansion $u_{\epsilon,1} \coloneqq u_0+\epsilon N_l^\epsilon \partial_l u_0$.

\begin{remark}
From the expression of $u_{\epsilon, 1}$, We can conjecture that twice differentiability of $u_0$ is the least regularity for $u_{\epsilon,1} \in H^1(\Omega)$, and $N_l(\bm{y}) \in W^{1,\infty}_\sharp(Q)$ is necessary for the square-integrability of $u_{\epsilon, 1}$. Actually, those two assumptions will frequently appear in our analysis.
\end{remark}

We have a theorem to justify the first-order expansion and the corresponding proofs can be found in \cite{Shen2016,Ye2019}:
\begin{theorem}\label{thm:ehalf}
Assume $u_0 \in H^2(\Omega)$ and $N_l(\bm{y}) \in W^{1,\infty}_\sharp(Q)$. For \cref{pro:a}, we have:
\[
\Seminorm{u_\epsilon - u_{\epsilon, 1}}_{1,\Omega} \leq C\Brackets{\kappa_1, \kappa_2, \Omega, \max_l\Norm{N_l}_{W^{1,\infty}_\sharp(Q)}} \epsilon^{1/2} \Norm{u_0}_{2, \Omega} ;
\]
for \cref{pro:b}:
\[
\vertiii{u_\epsilon - u_{\epsilon, 1}} \leq C\Brackets{\kappa_1, \kappa_2, \alpha_1, \alpha_2, \Omega, \max_l\Norm{N_l}_{W^{1,\infty}_\sharp(Q)}} \epsilon^{1/2} \Norm{u_0}_{2, \Omega} ;
\]
under \cref{ass:hemi}, for \cref{pro:c}:
\[
\Seminorm{u_\epsilon - u_{\epsilon, 1}}_{1,\Omega} \leq C\Brackets{\kappa_1, \kappa_2, \Delta, \Omega, \max_l\Norm{N_l}_{W^{1,\infty}_\sharp(Q)}} \epsilon^{1/2} \Norm{u_0}_{2, \Omega}.
\]
\end{theorem}

\subsection{Multiscale Finite Element Method}

Let $\mathcal{T}_h$ be a partition of $\Omega$ by triangles whose circumcircle radius $\leq h$. We need this partition be regular:
\begin{definition}[Regular triangulation family \cite{Brenner2008}]
For every element $T$ belonging to $\mathcal{T}_h$, denote $h_T$ and $r_T$ as its circumcircle and incircle radius respectively. The triangulation family is regular if there exists $\rho>0$ such that for all $T \in \mathcal{T}_h$ and for all $h \in (0, 1]$,
\[
\frac{r_T}{h_T} \geq \rho .
\]
\end{definition}

We only consider Lagrange element space $\mathcal{P}\coloneqq \{ u \in C(\overbar{\Omega}), \forall T \in \mathcal{T}_h, u|_T \in \mathcal{P}_1(T)  \}$, and take $\psi_j$ as a nodal basis function which vanishes at other nodes except at the $j$-th node. Correspondingly, the multiscale basis $\Phi_j$ can be constructed in such way:
\begin{equation}
\left\{
\begin{aligned}
&-\Div \Brackets{A(\bm{x}) \nabla \Phi_j} = 0 \text{~ in ~} T \\
&\Phi_j = \psi_j \text{~ on ~} \partial T
\end{aligned}.
\right. 
\end{equation}
This PDE determines $\Phi_j$ in each element $T$, and a simple observation tells that $\Phi_j$ is locally supported. Here we do not explicitly add $\epsilon$ on $A(\bm{x})$ because we want to emphasize MsFEM is a general method, and it is workable even \emph{without} a small periodicity assumption. Let $V_{\mathrm{ms}}$ be the finite linear space spanned by $\{\Phi_j\}$, then utilizing MsFEM to solve \cref{pro:a,pro:b,pro:c} is equivalent to solving the corresponding (hemi)variational forms on the finite dimensional function space $V_{\mathrm{ms}}$.

\section{Main Results}\label{sec:main}
\begin{lemma}
Let $A(\bm{x})$ be uniformly elliptic, and $\Omega$ be a Lipschitz domain in $\R^d$. Then $\forall u \in H^1(\Omega)$, we have:
\[
\Seminorm{u}_{1,\Omega} \leq C \Brackets{\sup_{\substack{\phi \in H^1_0(\Omega), \\ \phi \neq \bm{0}.}} \frac{\int_\Omega A \nabla u \cdot \nabla \phi}{\Seminorm{\phi}_{1,\Omega}} + \Seminorm{u}_{1/2, \partial \Omega}}.
\]
Here we redefine
\[
\Seminorm{g}_{1/2, \partial \Omega} = \inf \{ \Seminorm{v_g}_{1,\Omega} \colon v_g \in H^1(\Omega) \text{~ and ~} v_g = g \text{~ on ~} \partial \Omega\},
\]
and the constant $C$ only depends on $\kappa_1$, $\kappa_2$ of $A(\bm{x})$, and spacial dimension $d$.
\end{lemma}
\begin{remark}
We omit the proof since it is straight forward, and we emphasize the constant $C$ appears here does not involve the domain $\Omega$.
\end{remark}
\begin{lemma} \label{lem:criti}
Take $T$ as a triangle with $r$ as its incircle radius. Assume that correctors $\{N_l(\bm{y})\} \subset W^{1,\infty}_\sharp(Q)$, and $A(\bm{y})$ is 1-periodic, symmetric, uniformly elliptic. Let $w_0 \in \mathcal{P}_1(T)$ and $w_\epsilon$ be the solution of the PDE:
\[
\left \{
\begin{aligned}
& -\Div(A^\epsilon \nabla w_\epsilon) = 0 \text{~ in ~} T \\
& w_\epsilon = w_0 \text{~ on ~} \partial T
\end{aligned} 
\right. .
\]
Here $A^\epsilon(\bm{x})=A(\bm{x}/\epsilon)$. If $\epsilon \ll r$, then for the error of first-order expansion $w_{\epsilon,1}=w_0+\epsilon N_l^\epsilon \partial_l w_0$, we have
\[
\Seminorm{w_\epsilon - w_{\epsilon,1}}_{1, T} \leq C \sqrt{\frac{\epsilon}{r}} \Seminorm{w_0}_{1,T} ,
\]
here $C$ only depends on $\kappa_1, \kappa_2$ of $A(\bm{y})$ and $\max_l\Norm{N_l(\bm{y})}_{W^{1,\infty}_\sharp(Q)}$.
\end{lemma}

\begin{proof}
Let $\phi \in C^\infty_0(T)$, and $r_\epsilon = w_\epsilon - w_{\epsilon,1}$, we will have:
\[
\int_T A^\epsilon \nabla r_\epsilon \cdot \nabla \phi = \int_T A^\epsilon \nabla w_\epsilon \cdot \nabla \phi - \int_T \SquareBrackets{A^\epsilon_{il}+\Brackets{A_{ij}\partial_j N_l}^\epsilon}\partial_l w_0 \partial_i \phi
\]
It is obvious that $w_0$ is the homogenized solution of $w_\epsilon$, and this fact immediately leads $\int_T A^\epsilon \nabla w_\epsilon \cdot \nabla \phi = \int_T \hat{A} \nabla w_0 \cdot \nabla \phi$.  We can construct $\{E_{ijl}(\bm{y})\} \subset H^1_\sharp(Q)$, such that $\partial_j E_{ijl}(\bm{y})=\Brackets{\hat{A}_{il}-A_{ij}-A_{ij}\partial_j N_l}(\bm{y})$, and $E_{ijl}=-E_{jil}$. To see this, a classic proof by Fourier series analysis is provided in \cite{Jikov1994} page 27. There is a new proof (ref. \cite{Kenig2013a}) which we think is more ingenious and clear. We briefly introduce it here: if functions $\{b_{il}\} \subset L^2_\sharp(Q)$ satisfy $\partial_i b_{il} = 0$ and $\int_Q b_{il}(\bm{y}) \dx \bm{y}=0$; then define $G_{ijl}=\partial_i c_{jl}-\partial_j c_{il}$, where $-\Delta c_{il}=b_{il}$; one can verify $G_{ijl} \in H^1_\sharp(Q)$ with $\Norm{G_{ijl}}_{1,Q} \leq C\max_{i,l}\Norm{b_{il}}_{0,Q}$, $G_{ijl}=-G_{jil}$ and $b_{il}=\partial_j G_{ijl}$. Then
\[
\begin{aligned}
&\int_T A^\epsilon \nabla w_\epsilon \cdot \nabla \phi - \int_T \SquareBrackets{A^\epsilon_{il}+\Brackets{A_{ij}\partial_j N_l}^\epsilon}\partial_l w_0 \partial_i \phi = \int_T \SquareBrackets{\hat{A}-A^\epsilon_{il}-\Brackets{A_{ij}\partial_j N_l}^\epsilon}\partial_l w_0 \partial_i \phi \\
=&\int_T (\partial_j E_{ijl})^\epsilon \partial_l w_0 \partial_i \phi = \epsilon \int_T \partial_j (E_{ijl}^\epsilon) \partial_l w_0 \partial_i \phi \\
=& -\epsilon \int_T E_{ijl}^\epsilon \partial_l w_0 \partial_{ij} \phi = 0 .
\end{aligned}
\]
An integration by parts rule is used in the last line. We obtain $\int_T A^\epsilon \nabla r_\epsilon \cdot \nabla \phi = 0$, and we are left to estimate $\Seminorm{r_\epsilon}_{1/2, \partial T} = \Seminorm{\epsilon N_l^\epsilon \partial_l w_0}_{1/2, \partial T}$. Take $\mathrm{d}_T(\bm{x}) \coloneqq \mathrm{dist}\Brackets{\bm{x}, \partial T}$, and it is easy to show $\mathrm{d}_T(\bm{x})$ has a bounded gradient $\Seminorm{\nabla \mathrm{d}_T} \leq 1$. Here the triangle domain is simple enough to allow us to calculate $\mathrm{d}_T$, and refer \cite{LawrenceCraig2015} for more general discussions. Denote $T_\epsilon = \{ \bm{x} \in T \colon \mathrm{d}_T(\bm{x}) < \epsilon \}$, and $\theta_\epsilon \coloneqq (\epsilon - \mathrm{d}_T(\bm{x}))^+ / \epsilon$. We will see $0 \leq \theta \leq 1$ and $\epsilon N_l^\epsilon \partial_l w_0 \theta_\epsilon = \epsilon N_l^\epsilon \partial_l w_0 = -r_\epsilon$ on $\partial T$. Then we need to derive an estimation for the derivatives of $\epsilon N_l \partial_l w_0 \theta_\epsilon$:
\[
\begin{aligned}
\Norm{\partial_i \Brackets{\epsilon N_l^\epsilon \partial_l w_0 \theta_\epsilon}}_{0, T}^2 \leq & C \Brackets{\Norm{(\partial_i N_l)^\epsilon \partial_l w_0 \theta_\epsilon }_{0, T}^2 + \Norm{N_l^\epsilon \partial_l w_0 \partial_i (\epsilon \theta_\epsilon)}_{0,T}^2} \\
\leq &C \max_l \Brackets{\Norm{N_l}_{W^{1,\infty}_\sharp(Q)}^2 \Norm{\partial_l w_0}_{L^\infty(T)}^2} \Seminorm{T_\epsilon} \\
\leq & C \max_l \Norm{N_l}_{W^{1,\infty}_\sharp(Q)}^2 \Seminorm{w_0}_{1, T}^2 \frac{\Seminorm{T_\epsilon}}{\Seminorm{T}}
\end{aligned}.
\]
Here an important fact that $\partial_l w_0$ is constant in $T$ is considered, and the constant $C$ is independent of $T$. The precise value of $\Seminorm{T_\epsilon}/\Seminorm{T}$ is $1-(1-\epsilon/r)^2$, and we have $\Seminorm{T_\epsilon}/\Seminorm{T} \leq C \epsilon /r$ if $\epsilon$ is small enough comparing to $r$. 
\end{proof}

\begin{remark}
By a scaling argument, above lemma is a direct consequence of $\epsilon^{1/2}$ estimation (refer \cite{Shen2016,Jikov1994}). However, we must derive an estimation such that constants $C$ within inequalities are independent of the domain. Because the elements generated by meshing process are not always identical, typically in unstructured grids. This lemma can be naturally extended to a 3D domain. 
\end{remark}
We need an interpolation operator satisfied following properties, and one specific example is the local regularization operator in \cite{Bernardi1998}.
\begin{proposition} \label{prop:inter}
If the triangulation $\mathcal{T}_h$ is regular, then there exists an interpolation operator $\mathcal{I}:H^1(\Omega) \mapsto \mathcal{P}$ and a constant $C_{\mathcal{I}}$ such that $\forall u \in H^2(\Omega)$,
\[
\begin{aligned}
\Seminorm{u- \mathcal{I}u}_{1,\Omega} &\leq C_{\mathcal{I}} h \Seminorm{u_0}_{2,\Omega};\\
\Seminorm{\mathcal{I}u}_{1,\Omega} &\leq C_{\mathcal{I}} \Norm{u}_{2,\Omega};\\
\Norm{u - \mathcal{I} u}_{0, \partial \Omega} & \leq C_{\mathcal{I}} h^{3/2} \Norm{u}_{2,\Omega}.
\\
\end{aligned}
\]
Here the constant $C_\mathcal{I}$ depends on $\Omega$ and $\rho$ of $\mathcal{T}_h$.
\end{proposition}
Take $u_{0,\mathcal{I}} = \mathcal{I} u_0$, and $u_{\epsilon, \mathcal{I}} \in V_{\mathrm{ms}}$ with $u_{\epsilon, \mathcal{I}} = u_{0,\mathcal{I}}$ on each $\partial T$. Then we have a lemma:
\begin{lemma} \label{lem:uei}
Assume $u_0 \in H^2(\Omega)$ and $N_l(\bm{y}) \in W^{1,\infty}_\sharp(Q)$. For a regular triangulation $\mathcal{T}_h$ and an interpolation operator $\mathcal{I}$ that satisfies \cref{prop:inter}, we have:
\[
\Seminorm{u_{\epsilon, 1} - u_{\epsilon, \mathcal{I}}}_{1,\Omega} \leq C \Braces{\epsilon+h+\sqrt{\frac{\epsilon}{h}}}\Norm{u_0}_{2,\Omega},
\]
and
\[
\Norm{u_{\epsilon,1}-u_{\epsilon, \mathcal{I}}}_{0,\partial \Omega} \leq C\Braces{\epsilon+h^{3/2}} \Norm{u_0}_{2,\Omega}.
\]
Here the constant $C$ is independent of $\epsilon$ and $h$.
\end{lemma}
\begin{proof}
By a direct computation,
\[
\partial_i u_{\epsilon, 1} = \partial_i u_0 + \Brackets{\partial_i N_l}^\epsilon \partial_l u_0 + \epsilon N_l^\epsilon \partial_{il} u_0,
\]
and
\[
\Seminorm{u_{\epsilon, 1} - u_{\epsilon, \mathcal{I}}}_{1,\Omega}^2 \leq C \sum_i \Norm{\partial_i u_0 + \Brackets{\partial_i N_l}^\epsilon \partial_l u_0 - \partial_i u_{\epsilon, \mathcal{I}}}_{0, \Omega}^2 + \epsilon^2 \Norm{N_l}_{W^{1,\infty}_\sharp(Q)}^2 \Seminorm{u_0}_{2, \Omega}^2 .
\]
Then we focus on the term $\Norm{\partial_i u_0 + \Brackets{\partial_i N_l}^\epsilon \partial_l u_0 - \partial_i u_{\epsilon, \mathcal{I}}}_{0, \Omega}$. Insert $\partial_i u_{0, \mathcal{I}}$ and recall the assumption for interpolation $\mathcal{I}$, we have
\[
\begin{aligned}
&\Norm{\partial_i u_0 + \Brackets{\partial_i N_l}^\epsilon \partial_l u_0 - \partial_i u_{\epsilon, \mathcal{I}}}_{0, \Omega}^2 \\
\leq & C\Norm{\partial_i u_0- \partial_i u_{0, \mathcal{I}} + \Brackets{\partial_i N_l}^\epsilon \Brackets{\partial_l u_0-\partial_l u_{0, \mathcal{I}}}}_{0, \Omega}^2 + \int_\Omega \Seminorm{\partial_i u_{0, \mathcal{I}} +\Brackets{\partial_i N_l}^\epsilon\partial_l u_{0, \mathcal{I}} - \partial_i u_{\epsilon, \mathcal{I}}}^2 \\ 
\leq & C h^2 \Norm{N_l}_{W^{1,\infty}_\sharp(Q)}^2 \Norm{u_0}_{2, \Omega}^2 + C \sum_{T\in \mathcal{T}_h} \Seminorm{u_{\epsilon,\mathcal{I}}-u_{0,\mathcal{I}}-\epsilon N_l^\epsilon \partial_l u_{0, \mathcal{I}}}_{1, T}^2
\end{aligned}
\]
Note $u_{\epsilon, \mathcal{I}} = u_{0,\mathcal{I}}$ on every element's boundary, then combine \cref{lem:criti} and the regularity of $\mathcal{T}_h$,
\[
\sum_{T\in \mathcal{T}_h} \Seminorm{u_{\epsilon,\mathcal{I}}-u_{0,\mathcal{I}}-\epsilon N_l^\epsilon \partial_l u_{0, \mathcal{I}}}_{1, T}^2 \leq C \frac{\epsilon}{h} \Seminorm{u_{0,\mathcal{I}}}^2_{1,\Omega} \leq C \frac{\epsilon}{h} \Norm{u_0}_{2, \Omega}^2.
\]
The first inequality is established by summing together all these parts. For the second inequality, by definition of $u_{\epsilon, \mathcal{I}}$ we have $u_{\epsilon, \mathcal{I}} = u_{0, \mathcal{I}}$ on $\partial \Omega$, then
\[
\begin{aligned}
\Norm{u_{\epsilon,1}-u_{\epsilon, \mathcal{I}}}_{0,\partial \Omega}=&\Norm{u_{\epsilon,1}-u_{0, \mathcal{I}}}_{0,\partial \Omega} \\
\leq&\Norm{u_0-u_{0, \mathcal{I}}}_{0,\partial \Omega}+\epsilon \Norm{N_l^\epsilon \partial_l u_0}_{0,\partial \Omega} \\
\leq&C\Brackets{h^{3/2}+\epsilon} \Norm{u_0}_{2,\Omega}
\end{aligned}.
\]
Here a trace inequality is used in the last line.
\end{proof}
We take $u_{\epsilon, \mathrm{ms}}$ to represent the MsFEM solutions of \cref{pro:a,pro:b,pro:c}. For \cref{pro:a} and \cref{pro:b}, C\'{e}a's inequality can be directly utilized.
\begin{theorem} \label{thm:ab}
Under the same assumptions in \cref{lem:uei}, for \cref{pro:a} we have:
\[
\Seminorm{u_\epsilon-u_{\epsilon, \mathrm{ms}}}_{1,\Omega} \leq C\Braces{\epsilon^{1/2}+h+\sqrt{\frac{\epsilon}{h}}} \Norm{u_0}_{2,\Omega}.
\]
For \cref{pro:b},
\[
\vertiii{u_\epsilon- u_{\epsilon, \mathrm{ms}}} \leq C\Braces{\epsilon^{1/2}+h+\sqrt{\frac{\epsilon}{h}}} \Norm{u_0}_{2,\Omega}.
\]
\end{theorem}
\begin{proof}
By C\'{e}a's inequality, for \cref{pro:a}
\[
\Seminorm{u_\epsilon-u_{\epsilon, \mathrm{ms}}}_{1,\Omega} \leq C\Seminorm{u_\epsilon-u_{\epsilon, \mathcal{I}}}_{1,\Omega} \leq C \Brackets{\Seminorm{u_\epsilon-u_{\epsilon, 1}}_{1,\Omega}+\Seminorm{u_{\epsilon, 1}-u_{\epsilon, \mathcal{I}}}_{1,\Omega}} ,
\]
and the result follows from \cref{thm:ehalf} and \cref{lem:uei} by discarding high order terms, similarly for \cref{pro:b}.
\end{proof}

C\'{e}a's inequality for hemivariational problem is slightly different (cf. \cite{Han2017}), and from the reference we can obtain
\[
\Seminorm{u_\epsilon-u_{\epsilon, \mathrm{ms}}}_{1,\Omega}^2 \leq C \Brackets{\Seminorm{u_\epsilon-u_{\epsilon, \mathcal{I}}}_{1,\Omega}^2 + \Norm{u_\epsilon-u_{\epsilon, \mathcal{I}}}_{0,\Gamma_C}}.
\]
Then we have a theorem:
\begin{theorem}\label{thm:c}
Under the same assumptions in \cref{lem:uei}, for \cref{pro:c},
\[
\Seminorm{u_\epsilon-u_{\epsilon, \mathrm{ms}}}_{1,\Omega} \leq C\Braces{\Brackets{\epsilon^{1/2}+h+\sqrt{\frac{\epsilon}{h}}} \Norm{u_0}_{2,\Omega}+\sqrt{R_\epsilon}+h^{3/4} \sqrt{\Norm{u_0}_{2,\Omega}}}
\]
Here $R_\epsilon = \Norm{u_\epsilon-u_0}_{0,\Gamma_C}$.
\end{theorem}
\begin{remark}
The proof here is clear by inserting $u_0$ into $\Norm{u_\epsilon-u_{\epsilon, \mathcal{I}}}_{0,\Gamma_C}$, and we also note by trace inequality, $R_\epsilon \leq C \Seminorm{u_\epsilon-u_{\epsilon, 1}}_{1,\Omega} + \epsilon \Norm{N_l^\epsilon \partial_l u_0}_{0, \Gamma_C} \leq C \epsilon^{1/2}\Norm{u_0}_{2,\Omega}$. However, this estimation for $R_\epsilon$ is suboptimal, and we conjecture that the optimal result is $R_\epsilon=O(\epsilon)$. Compare \cref{thm:ab,thm:c} with the first-order expansion method in \cite{Ye2019}, a resonance error $\sqrt{\epsilon/h}$ is inevitably induced because of neglecting the oscillation on inner element boundaries, and this fact triggers considerable subsequent works.
\end{remark}

In \cite{Hou1999}, the authors formally stated a sharp $L^2$-estimation without proof. Actually, since the right-hand section of our estimation in \cref{thm:ab} only depends on the $H^2$ norm of $u_0$, we can directly prove an $L^2$-estimation by the Aubin-Nitsche technique.
\begin{corollary}
Under the same assumptions in \cref{lem:criti} and suppose the domain is regular (that is, $\exists C>0, \forall F \in L^2(\Omega), \exists v_0\in H^1_{0,\Gamma_D}(\Omega) \cap H^2(\Omega)$ s.t. $\forall \phi \in H^1_{0,\Gamma_D}(\Omega)$, $\int_\Omega \hat{A}\nabla v_0\cdot \nabla \phi=\int_\Omega F \phi$ and $\Norm{v_0}_{2,\Omega} \leq C\Norm{F}_{0,\Omega}$.). For \cref{pro:a}, we have:
\[
\Norm{u_0-u_{\epsilon, \mathrm{ms}}}_{0,\Omega} \leq C\Braces{\epsilon+h^2+\frac{\epsilon}{h}} \Norm{u_0}_{2,\Omega}.
\] 
\end{corollary}
\begin{proof}
Take any $F\in L^2(\Omega)$, and let $v_{\epsilon,\mathrm{ms}}$ be the MsFEM solution of the PDE:
\[
\left\{
\begin{aligned}
-\Div\Brackets{A^\epsilon \nabla v_\epsilon} = F &\text{~ in ~} \Omega \\
v_\epsilon = 0 &\text{~ on ~} \Gamma_D \\
\bm{n} \cdot A^\epsilon \nabla v_\epsilon = 0 &\text{~ on ~} \Gamma_N
\end{aligned}
\right. .
\]
Then by the Galerkin orthogonality and \cref{thm:ab},
\[
\begin{aligned}
\int_\Omega F \cdot (u_\epsilon-u_{\epsilon,\mathrm{ms}}) &= \int_\Omega A^\epsilon \nabla v_\epsilon \cdot \nabla u_\epsilon - \int_\Omega A^\epsilon \nabla v_{\epsilon,\mathrm{ms}} \cdot \nabla u_{\epsilon,\mathrm{ms}} = \int_\Omega A^\epsilon \nabla (v_\epsilon-v_{\epsilon,\mathrm{ms}}) \cdot \nabla (u_\epsilon-u_{\epsilon,\mathrm{ms}}) \\
&\leq C(\epsilon^{1/2}+h+\sqrt{\epsilon/h}) \Norm{v_0}_{2,\Omega} (\epsilon^{1/2}+h+\sqrt{\epsilon/h}) \Norm{u_0}_{2,\Omega} \\
&\leq C(\epsilon+h^2+\epsilon/h) \Norm{u_0}_{2,\Omega} \Norm{F}_{0,\Omega}
\end{aligned}.
\]
This reveals $\Norm{u_\epsilon-u_{\epsilon,\mathrm{ms}}}_{0,\Omega} \leq C(\epsilon+h^2+\epsilon/h) \Norm{u_0}_{2,\Omega}$. We claim the conclusion by the fact that $\Norm{u_\epsilon-u_0}_{0,\Omega}\leq C\epsilon\Norm{u_0}_{1,\Omega}$.
\end{proof}

\section{Conclusion}\label{sec:conclusion}
In this paper, we proposed a proof for the convergence rate of MsFEM for mixed Dirichlet-Neumann, Robin and hemivariational inequality boundary problems. The key step is directly utilizing the first-order expansion rather than constructing boundary correctors, and this also leads a relaxation on assumptions. Our proof also indicates the originality of resonance error, while whether on those problems this ``annoying'' error can be reduced by over-sampling technique is still under considering. Due to the nonlinear setting, the estimation for hemivariational inequality problem is slightly verbose comparing to linear problems, and a refined result relies on further developments of homogenization theory. 

We admit that with numerical experiments implemented our theoretical analysis will be more plausible. However, since MsFEM itself has evolved considerably and its variants have been broadly applied to different areas, it seems that classic MsFEM is ``undoubtedly'' effective on various boundary problems. Our work may strengthen the understanding of how MsFEM captures small-scale information in small periodicity setting. The reason is that it successfully approximates the first-order expansion which is related to more general correctors in H-convergence theory (\cite{Tartar2010}).

\section{Acknowledgment}
This work is supported by the National Natural Science Foundation of China [51739007], the State Key Laboratory of Scientific and Engineering Computing, the China Postdoctoral Science Foundation (No. 2018M643573), and the National Natural Science Foundation of Shaanxi Province (No. 2019JQ-048). We also thank the referees for their comments and suggestions, which have lead to a significant improvement of this paper. Moreover, The author C. Ye owes his special thanks to Ms. Lin Pan for her consistent pushing and continuous support.

\bibliographystyle{siamplain}
\bibliography{references}    
\end{document}